\documentclass[12pt]{article}

\usepackage[a4paper,left=2.5cm,top=3cm,textwidth=16cm,textheight=24.5cm]{geometry}

\usepackage[bookmarksopen=true,bookmarks=true,unicode,setpagesize]{hyperref}
\hypersetup{colorlinks=true,linkcolor=black,citecolor=black}

\usepackage[intlimits]{amsmath}

\usepackage{amsthm,amssymb}

\theoremstyle{plain}
\newtheorem{theorem}{Theorem}
\newtheorem{lemma}[theorem]{Lemma}

\theoremstyle{definition}

\newtheorem{example}[theorem]{Example}

\theoremstyle{remark}
\newtheorem{remark}[theorem]{Remark}

\allowdisplaybreaks[3]

\newcommand\N{\mathbb{N}}
\newcommand\R{\mathbb{R}}

\newcommand\X{{\R^d}}

\begin{document}

\pagestyle{plain}
\title{Perpetual Integral Functionals of Multidimensional Stochastic Processes  }

\date{\today}
\author{
\textbf{Yuri Kondratiev}\\
 Department of Mathematics, University of Bielefeld, \\
 D-33615 Bielefeld, Germany,\\
 Dragomanov University, Kyiv, Ukraine\\
 e-mail: kondrat@mathematik.uni-bielefeld.de\\
 Email: kondrat@math.uni-bielefeld.de
 \and \textbf{Yuliya Mishura}\\
 Taras Shevchenko National University of Kyiv\\
e-mail: myus@univ.kiev.ua
\and\textbf{ Jos{\'e} L.~da Silva},\\
 CIMA, University of Madeira, Campus da Penteada,\\
 9020-105 Funchal, Portugal.\\
e-mail: joses@staff.uma.pt}

\maketitle
\begin{abstract}
The paper is devoted to the existence of integral functionals $\int_0^\infty f(X(t))\,{\mathrm{d}t}$ for several classes of  processes in $\X$ with  $d\ge 3$.   Some examples such as Brownian motion, fractional Brownian motion, compound Poisson process,   Markov processes admitting densities of transitional probabilities     are considered.

{\em Keywords:} Perpetual integral functionals, Markov processes, fractional Brownian motion, Compound Poisson process.

{\em AMS Subject Classification 2010:} 60J25, 60J65, 60G22, 47A30.
\end{abstract}

\tableofcontents{}

\section{Introduction}
\label{sec:Introduction}
Let  $X=\{X(t), t\ge 0\}$ be a $d$-dimensional stochastic process with c\`{a}dl\`{a}g  trajectories, and let $f:\X\rightarrow \R$ be a continuous measurable function. Then for any $T>0$ the integral functional $\int_0^T f(X_t)\,{\mathrm{d}t}$ is defined. However, its properties and asymptotic behavior as $T\rightarrow \infty$ depend crucially on the properties of process $X$ and the dimension $d$. In particular, the integral functionals of one-dimensional Brownian motion (Bm for short) $B$ received a lot of attention, which for functions $f\in L_1(\R)$ yields
$$
\int_0^T f(B(t))\,{\mathrm{d}t}=\int_\R f(x)L_T(x)\,{\mathrm{d}x},
$$
where $L_T(x)$ is the local time of Bm up to moment $T$ at the point $x$. For the definitions and properties of local time of one-dimensional Bm, see, e.g., \cite{Bor, Bor1, Tak, Tak1}. Now, the asymptotic behavior of the integral functional  $\int_0^T f(X(t))\,{\mathrm{d}t}$ is very different  even for one-dimensional Markov processes and depends on their transient or recurrent properties. For example, one-dimensional   Bm $B$ is recurrent, therefore $L_\infty(x)=\infty$ for all $x\in \R$, consequently, the integral functional $\int_0^\infty f(B(t))\,{\mathrm{d}t}$, roughly speaking, does not exist. Contrary to this situation, one-dimensional Bm with positive drift, $B^\mu_t=B(t)+\mu t$, $\mu>0$  is transient therefore the perpetual integral functional  $\int_0^\infty f(B^\mu_t)\,{\mathrm{d}t}$ is finite for any nonnegative locally integrable function $f$, integrable at infinity (see \cite{Salm}). The asymptotic behavior of the integral functional $\int_0^T f(X(t))\,{\mathrm{d}t}$ and respective normalization was established in \cite{KMS}. Concerning two-dimensional processes, a lot of articles are devoted to the self-intersection local time, mainly for planar Brownian motion, see, e.g., \cite{Bass}.

Now, consider a $d$-dimensional stochastic processes $X(t)$, $t\ge 0$ with  $d>2$. As the dimension grows, the situation changes. In particular, $d$-dimensional Bm becomes transient, and it leads to the existence of the perpetual integral functional $\int_0^\infty f(B(t))\,{\mathrm{d}t}$ for a sufficiently large class of functions $f$.  However, it is reasonable to consider wider classes of stochastic processes  for which the perpetual integral functional $\int_0^\infty f(X(t))\,{\mathrm{d}t}$ exists a.s.  A natural and  simple sufficient condition for the existence is to have a finite expectation. In this connection, a natural class of function $f$ consists of bounded continuous and integrable functions.
The expectation
$$
u_f(x) = E^x\left[\int_0^\infty f(X(t))\,{\mathrm{d}t}\right]
$$
is well-known as the potential of the function $f$,   see, e.g., \cite{GB}.
The question concerning the description of admissible functions for a given
process for which the potential exists is rather open. A bit more simple situation
we have in the case of Markov processes. If $L$ is the generator of a Markov process
$X(t)$, $t\ge 0$ then $u_f$  is the solution of the following equation
$$
-Lu=f.
$$
As in the classical PDE theory we would like to write this solution in the
form
$$
u_f(x) = \int_{\X} f(y) \mu(x,{\mathrm{d}y}),
$$
where $\mu(x,{\mathrm{d}y})$ is the fundamental solution (measure)  corresponding to the operator
$L$. In the simplest cases as the Laplace operator $L=\Delta$ it is the Green function
for $\Delta$
and we will call $\mu(x,{\mathrm{d}y})$ the Green measure for the process \cite{KS}.  Of course, the notion
of a Green measure in the integral representation for the potential may be introduced
without any Markov property.  We would like to stress
that the existence and properties of Green measures are highly depending of the class
of processes under consideration. From this point of view, the perpetual integrals can be called random potentials, although the concept of potential in stochastics is used for another object.
Furthermore, the existence of the perpetual integral functional immediately implies the existence of   occupation measure of stochastic process $X$ on $[0,\infty)$ that is defined as
$$\mu_X(A, \omega)=\lambda(X^{-1}(A))=\lambda(t\in [0,\infty): X_t\in A),$$
for any Borel sets $A\in\X$, and the formula of change of measure leads to the following representation
 $$
\left(\int_0^\infty f(X(t))\,{\mathrm{d}t}\right)(\omega)= \int_\X f(y) \mu(x,{\mathrm{d}y},\omega).
$$
If we follow the "Green terminology", occupation measure can be called random Green measure.
Generally speaking, the existence of occupation measure does not mean that the integral
$\int_0^\infty f(X(t))\,{\mathrm{d}t}$ is a non-degenerate random variable, i.e., is not a constant. However, we establish that for Bm and fractional Brownian motion (fBm) this integral is not a constant whenever it exists and the function $f$ is not identical zero.  The proof is based on the property of conditional full support for the distribution of these processes.

\section{Existence and Representation of the Perpetual Integral Functional}
\label{sec:Existence-PIF}
Let    $X=\{X(t), t\ge 0\}$ be a $d$-dimensional stochastic process with c\`{a}dl\`{a}g  trajectories, and let  $f:\X\rightarrow \R$ be  a measurable and continuous function. Without loss of generality, we assume throughout the paper that the function $f$ is non-negative. From now on the following two assumptions are satisfied
\begin{equation}\label{A}
\tag{A}\int_0^\infty E[f(X(t))]\,{\mathrm{d}t}<\infty.
\end{equation}
and
\begin{equation}
\tag{B} |X(t)| \to \infty, t\to\infty,\; \mathrm{a.s.}
\end{equation}

Under condition $(A)$, according to the standard Fubini theorem, the perpetual integral functional
$\int_0^\infty  f(X_t)\,{\mathrm{d}t}$ exists with probability 1. Applying the change of variables,
we can write
$$
\int_0^\infty  f(X_t(\omega))\,{\mathrm{d}t}=\int_{\X}f(x)  \mu_X(\mathrm{d}x, \omega),
$$
where $\mu_X({\mathrm{d}x}, \omega)$ is the occupation measure (for a.a.~$\omega$)) of the process $X$ on $[0,\infty)$ defined as
 $$\mu_X(A, \omega)=\lambda(X^{-1}(A))=\lambda(t\in [0,\infty): X_t\in A),$$ 
for every Borel set  $A\in \X$.
Thanks to condition (B), the occupation measure $\mu_X$ is a.s.\ locally finite, so it is a Radon measure as locally finite Borel measure on a Polish  space.

\section{Examples of $d$-dimensional Processes Admitting Perpetual Integral Functionals}
\label{sec:Examples}
\subsection{Brownian Motion}
Denote $B=\{B(t),t\ge 0\}$ the Bm in $\X$ starting from the point $x\in\X$. It is a Markov process with generator $\Delta$. We are interested in the conditions under which the perpetual integral functional
 $Y(f) =\int_0^\infty f(B(t))\,{\mathrm{d}t} $ does exists
for a certain class of functions $f:\X\to \R$. As mention above, we consider non-negative functions $f$. Introduce the following class of functions
$$CL(\X)=\{f:\X\rightarrow \R: f\; \text{is continuous, bounded and belongs to}\; L_1(\X) \}.$$
It is a Banach space with the norm $\|f\|_{CL}:=\sup|f| +\|f\|_{L_1(\X)}$.

 \begin{theorem}\label{theo1}  For any $f\in CL(\X)$ the perpetual integral functional $\int_0^\infty f(B(t))\,{\mathrm{d}t}$ exists a.s., its  expectation equals
 \begin{equation}
 \begin{gathered}\label{eq:exp-BM}
 E\left[\int_0^\infty f(B(t))\,{\mathrm{d}t}\right]=\frac{2^{d/2-1}\Gamma\left(d/2-1\right)}{(2\pi)^{d/2}}\int_{\X}\frac{f(x+y) }{|y|^{d-2}}\,{\mathrm{d}y},
 \end{gathered}
 \end{equation}
  whereas  variance equals
 \begin{equation}
 \begin{gathered}\label{eq:var}
 V(f):=E\left[\left(\int_0^\infty f(B(t))\,{\mathrm{d}t}-E\left[\int_0^\infty f(B(t))\,{\mathrm{d}t}\right]\right)^2\right]\\=\frac{2^{d-2}\Gamma^2(d/2-1)}{(2\pi)^d} \left(2\int_{\X}\int_{\X}\frac{f(x+y)f(x+y+z)}{|y|^{d-2}|z|^{d-2}}\,{\mathrm{d}y}\,{\mathrm{d}z} -  \left(\int_{\X} \frac{f(x+y)} {|y|^{d-2}}\,{\mathrm{d}y}\right)^2 \right).
 \end{gathered}
 \end{equation}
 For $f\neq0$, $\int_0^\infty f(B(t))\,{\mathrm{d}t}$ is a non-constant random variable, consequently for $f\neq0$ the right-hand side of   \eqref{eq:var} is strictly positive.
\end{theorem}

\begin{proof} Concerning the expectation, for non-negative $f$ we can apply Fubini theorem and obtain
\begin{equation*}\begin{gathered}
E\left[\int_0^\infty f(B(t))\,{\mathrm{d}t}\right]=\int_0^\infty\int_{\X}f(x+y) \left(2\pi t\right)^{-d/2}\exp\left\{-\frac{|y|^2}{2t}\right\}\,{\mathrm{d}y}\,{\mathrm{d}t}\\=
\int_{\X}f(x)\bigg( \int_0^\infty\left(2\pi t\right)^{-d/2}\exp\left\{-\frac{|y|^2}{2t}\right\}\,{\mathrm{d}t}\bigg)\,{\mathrm{d}x}.
\end{gathered}\end{equation*}
Now, after elementary calculations the inner integral  can be transformed as follows:
\begin{equation}\label{eq:int}
\int_{0}^\infty(2\pi t)^{-d/2}\exp\left\{-\frac{|y|^2}{2t}\right\}\,{\mathrm{d}t}=\frac{2^{d/2-1}\Gamma\left(d/2-1\right)}{(2\pi)^{d/2}|y|^{d-2}},
\end{equation}
which gives \eqref{eq:exp-BM}. The integral on the right-hand side of \eqref{eq:exp-BM} is finite. In fact, we may use the local integrability of $|y|^{2-d}$ in $y$ and conclude that
\begin{equation}\begin{gathered}\label{upperbound}
  \int_{\X} \frac{f(x+y)}{| y|^{d-2}}\,{\mathrm{d}y} \leq  \int_{| y|\leq 1} \frac{f(x+y)}{| y|^{d-2}}\,{\mathrm{d}y}
+   \int_{| y|>1} \frac{f(x+y)}{| y|^{d-2}}\,{\mathrm{d}y} \\ \leq
 C_1 \|f\|_\infty + C_2 \|f\|_1 \leq C\|f\|_{CL},
 \end{gathered}\end{equation}
  so, the integral in \eqref{eq:exp-BM} is indeed correctly defined. It means that $\int_0^\infty f(B(t))\,{\mathrm{d}t}$ exists with probability 1. Concerning its variance, let $0<s<u$. Taking into account that $B$ has  independent increments, so $B(s)$ and $B(u)-B(s)$ are independent, we claim that the following equality holds
\begin{equation}
\begin{gathered}\label{2nd-mom-point}
E[f(B(u))f(B(s))]=E[f(B(u)-B(s)+B(s))f(B(s))]\\=\int_{\X}\int_{\X}f(x+y)f(x+y+z)\frac{\exp\left\{-\frac{|y|^2}{2s}\right\}}{(2\pi s)^{d/2}}\frac{\exp\left\{-\frac{|z|^2}{2(u-s)}\right\}}{(2\pi (u-s))^{d/2}}\,{\mathrm{d}y}\,{\mathrm{d}z}.
\end{gathered}\end{equation}
Therefore we immediately get that
\begin{equation}
\begin{gathered}\label{2nd-mom-contin}
E\left[\left(\int_0^{\infty} f(B(t))\,{\mathrm{d}t}\right)^2\right]=\int_{0}^{\infty}\int_{0}^\infty E\left[f(B(u))f(B(s))\right]\,{\mathrm{d}u}\,{\mathrm{d}s}=2\int_{0}^\infty\int_{s}^\infty E\left[f(B(u))f(B(s))\right]\,{\mathrm{d}u}\,{\mathrm{d}s}\\
=2\int_{\X}\int_{\X}f(x+y)f(x+y+z)\int_{0}^{\infty}\int_{s}^\infty\frac{\exp\left\{-\frac{|y|^2}{2s}\right\}}{(2\pi s)^{d/2}}\frac{\exp\left\{-\frac{|z|^2}{2(u-s)}\right\}}{(2\pi (u-s))^{d/2}}\,{\mathrm{d}u}\,{\mathrm{d}s}\,{\mathrm{d}y}\,{\mathrm{d}z}\\=
2\int_{\X}\int_{\X}f(x+y)f(x+y+z)\int_{0}^{\infty}\int_{0}^\infty\frac{\exp\left\{-\frac{|y|^2}{2s}\right\}}{(2\pi s)^{d/2}}\frac{\exp\left\{-\frac{|z|^2}{2u}\right\}}{(2\pi u)^{d/2}}\,{\mathrm{d}u}\,{\mathrm{d}s}\,{\mathrm{d}y}\,{\mathrm{d}z}.
\end{gathered}
\end{equation}
 Now, \eqref{eq:var}   follows from   \eqref{2nd-mom-contin}  combined with  the following simple observation:  for non-negative function $f\in CL(\X)$
$$\int_{\X}\int_{\X}\frac{f(x+y)f(x+y+z)}{|y|^{d-2}|z|^{d-2}}\,{\mathrm{d}y}\,{\mathrm{d}z}\le \left( C_1 \|f\|_\infty + C_2 \|f\|_1\right)^2\leq (C\|f\|_{CL})^2.$$
The most interesting part of the proof is to establish that for $f\neq0$, $\int_0^\infty f(B(t))\,{\mathrm{d}t}$ is a non-constant random variable. Consider any continuous stochastic process $X$ and without loss of generality assume that $X(0)=0.$ Let $t_0>0$ and $z\in C([0,t_0])$ be given and let
$P_{z, t_0}(u), u\in C([t_0,\infty))$ be the regular conditional distribution of $\{X(t), t\in [0,t_0]\}$ under the condition
$\{(X(t), t\ge t_0)=u\}.$ Assume that for a.a.~$u\in C([t_0,\infty))$
\begin{equation}\label{eq:supp} \operatorname{supp} P_{z, t_0}(u)=\{y\in C([0,t_0]): z(t_0)=y(t_0)\}. \end{equation}
Now, assume  the contrary, that is, there is a constant $c>0$ and a  function $f\in C(\R^d)$ such that
$$
A:=\int_0^\infty f(X(t))\,{\mathrm{d}t}=c,\quad \mathrm{a.s.}
$$
Since  $A=\int_0^{t_0} f(X(t))\,{\mathrm{d}t}+ \int^\infty_{t_0} f(X(t))\,{\mathrm{d}t}$, we can write that
\begin{equation*}0= P(A\neq c)=\int_{C([t_0,\infty))}P_{z, t_0}\left(\left\{\int_0^{t_0} f(X(t))\,{\mathrm{d}t}+\int^\infty_{t_0} f(X(t))\,{\mathrm{d}t} \neq c\right\},u\right)P({\mathrm{d} u}), \end{equation*}
we conclude that
$$
P_{z, t_0}\left(\left\{\int_0^{t_0} f(X(t))\,{\mathrm{d}t}+\int^\infty_{t_0} f(X(t))\,{\mathrm{d}t} \neq c\right\},u\right)=0
$$
for a.a.~$u\in C([t_0,\infty))$. Now,  from \eqref{eq:supp} combined with the  continuity of the map $z\rightarrow \int_0^{t_0} f(z(t))\,{\mathrm{d}t}$ in the sup-norm, it follows that
\begin{equation}
\int_0^{t_0} f(v(t))\,{\mathrm{d}t}+\int^\infty_{t_0} f(u(t))\,{\mathrm{d}t}=c
\end{equation}
for any $v\in C([0,\infty))$ such that $v(0)=0, v(t_0)=u(t_0)$ and for a.a.~$u\in C([t_0,\infty])$. Let us fix such $t_0$
and consider the  functions
$$z_\varepsilon(t)=\begin{cases}
\frac{at}{\varepsilon} , &t\in [0,\varepsilon],\\
a, &t\in (\varepsilon,t_0-\varepsilon)\\
\frac{u(t_0)(t-t_0+\varepsilon)+a(t_0-t)}{\varepsilon}, &t\in [t_0-\varepsilon,t_0].
\end{cases}
$$
Then we get   the equality
$$\int_0^{t_0}f(z_\varepsilon(t))\,{\mathrm{d}t}+\int_{t_0}^\infty f(u_0(t))\,{\mathrm{d}t}=c,$$
and  tending $\varepsilon\rightarrow 0,$
we conclude that $$t_0f(a)+\int_{t_0}^\infty f(u_0(t))\,{\mathrm{d}t}=c.$$
Hence, $$ f(a)=\frac{c-\int_{t_0}^\infty f(u_0(t))\,{\mathrm{d}t}}{t_0} ,$$
which means that $f$ is a constant consequently zero function. Note that assumption \eqref{eq:supp} is fulfilled for Brownian motion, whence the proof follows.
\end{proof}

\subsection{Fractional Brownian Motion}

Consider a $d$-dimensional fBm with Hurst parameter $H\in(0,1)$, namely, $B^H(t)=(B_1^H(t),\ldots,B_d^H(t)),$ where all coordinates $B_i^H$ are independent 1-dimensional fBms, i.e., Gaussian processes with zero mean and covariance function
$$
E[B_i^H(t)B_i^H(s)]=\frac{1}{2}\big(t^{2H}+s^{2H}-|t-s|^{2H}\big).
$$

As above we introduce the perpetual integral functional $ \int_0^\infty f(x+B^H(t))\,{\mathrm{d}t}$.

\begin{theorem}
Let $d>1/H $ and $x\in\X$ be given. Then for any $f\in CL(\X)$ the perpetual integral functional $\int_0^\infty f(x+B(t))\,{\mathrm{d}t}$ exists a.s., its  expectation equals
\begin{equation}
\begin{gathered}\label{fract}
E\left[\int_0^\infty f(x+B^H(t))\,{\mathrm{d}t}\right]=C_{d,H}\int_{\X}\frac{f(x+y)}{|y|^{d-1/H}}\,{\mathrm{d}y},
\end{gathered}
\end{equation}
where
$$
C_{d,H}=2^{-(1+1/(2H))}H^{-1}\pi^{-d/2}\Gamma\left(\frac{d}{2}-\frac{1}{2H}\right).
$$
 For $f\neq0$, $\int_0^\infty f(x+B^H(t))\,{\mathrm{d}t}$ is a non-constant random variable consequently for $f\neq0$ the variance of $\int_0^\infty f(x+B^H(t))\,{\mathrm{d}t}$ is strictly positive.

\end{theorem}

\begin{proof} The density of distribution of $B^H(t)$ equals $(2\pi t^{2H})^{-d/2}e^{-|x|^2/(2t^{2H})}$.
Then equality  \eqref{fract} and the upper bound  for  $E\big[\int_0^\infty f(x+B^H(t))\,{\mathrm{d}t}\big]$ are established in the same lines as the proof of Theorem \ref{theo1}. Concerning the fact that for $f\neq0$, $\int_0^\infty f(x+B^H(t))\,{\mathrm{d}t}$ is a non-constant random variable, for any function $z: (0,\infty) \rightarrow \X$ define $R(z)(t)=t^{2H}z(1/t).$  It is easy to see that $R(z)^2$ is an identical transformation and $R(B^H)(t),t\ge 0$ has the same distribution as $\{B^H(t),t\ge 0\}$.
Then for any $t_0>0$ the distribution of $\{B^H(t),t\in[0,t_0]\}$ given $\{B^H(t),t\ge t_0\}=:u\in C([t_0,\infty))$ is the same as the image under $R$ of the conditional distribution of $\{B^H(t),t\ge 1/t_0\}$ given $\{B^H(t),0<t\le 1/t_0\}=R(u).$ Since the latter has full support, see the paper \cite{Chern}, the proof follows as the respective proof in Theorem \ref{theo1}.
\end{proof}

\subsection{Compound Poisson Process}
Let $\xi_k$, $k\ge 1$ be a sequence of iid $d$-dimensional random variables, and $N=\{N(t), t\ge 0\}$ a homogeneous Poisson process with intensity $\lambda>0$, independent of $\xi_k, k\ge 1$. Denote $X=\{X(t),t\ge 0\}$ the compound Poisson process  in $\X$ starting from zero, i.e.,
$$
X(t)=\sum_{k=1}^{N(t)}\xi_k,
$$
where we use the convention that $\sum_{k=1}^{0}=0.$ The process $X(t)$, $t\ge 0$ has independent increments consequently a Markov process, and if we assume that any $\xi_k$ has probability density $a$ and that $\lambda=1$, then its  generator is defined on $CL(\X)$ and equals $$
Lf(x)= \int_{\X} a(x-y)[f(y)-f(x)]\,{\mathrm{d}y}.
$$  We are interested on the  conditions under which the perpetual integral functional
 $Y(f) =\int_0^\infty f(X(t))\,{\mathrm{d}t} $ does exists
for  non-negative $f\in CL(\X)$.  In this connection, our goal now is to check condition \eqref{A}.  Denote $a_k(x)=a^{*k}(x)$ the $k$-fold convolution of the density $a$ and let
$$
G_0 (x)= \sum_{k=1}^\infty {a_k(x)},
$$
provided that this series converges for any $x\in \X$.
\begin{theorem}\label{theo3} Assume that any $\xi_k$ has probability density $a$ and   $\lambda=1$. Assume also that  $G_0 (x)$ is integrable in some ball $B(0,R)$ and bounded outside this ball. Then for any $f\in CL(\X)$ and every  $x\in \X$ the perpetual integral functional $\int_0^\infty f(x+X(t))\,{\mathrm{d}t}$ exists a.s., its  expectation equals \begin{equation}
\begin{gathered}\label{eq:exp}
E\left[\int_0^\infty f(x+X(t))\,{\mathrm{d}t}\right]= f(x)+\int_{\X}  f(x+y) G_0(y)\,{\mathrm{d}y},\end{gathered}\end{equation}
  whereas  variance equals
 \begin{equation}
 \begin{gathered}\label{eq:var1}
 V(f):=E\left[\left(\int_0^\infty f(x+X(t))\,{\mathrm{d}t}-E\left[\int_0^\infty f(x+X(t))\,{\mathrm{d}t}\right]\right)^2\right]\\= 2\int_{\X}\int_{\X} f(x+y)f(x+y+z)  G_0(y)G_0(z)\,{\mathrm{d}y}\,{\mathrm{d}z} -  \left(\int_{\X}  f(x+y)G_0(y)\,{\mathrm{d}y}\right)^2.
 \end{gathered}\end{equation}

 \end{theorem}
 \begin{proof}   It follows from the independence of $\xi$ and $N$ that
 \begin{equation*}
 \begin{gathered}
 E[f(x+X(t))]=\sum_{n=0}^\infty E\left[f\left(x+\sum_{k=1}^n \xi_k\right)\right]P(N(t)=n)\\
 =f(x)e^{-t}+\sum_{n=1}^\infty e^{-t}\frac{t^n}{n!}\int_{\X}f(x+y)a_n(y)\,{\mathrm{d}y}.
 \end{gathered}\end{equation*}
 Hence, formally
 \begin{equation*}\begin{gathered}
 E\left[\int_0^\infty f(x+X(t))\,{\mathrm{d}t}\right]= f(x)+\int_{\X}\sum_{n=1}^\infty f(x+y)a_n(y)\,{\mathrm{d}y}\int_0^\infty e^{-t}\frac{t^n}{n!}\,{\mathrm{d}t}\\
 =f(x)+\int_{\X}  f(x+y)G_0(y)\,{\mathrm{d}y},
 \end{gathered}\end{equation*}
 and under the assumption that  $G_0 (y)$ is integrable in some ball $B(0,R)$ and bounded outside this ball, we obtain (similarly to \eqref{upperbound}) the existence of the expectation. Further calculations are also similar to the respective calculations in the proof of Theorem \ref{theo1}.
 \end{proof}
 Let us provide one simple sufficient condition for the boundedness and integrability of $G_0$. Taking into account the condition in Theorem \ref{theo3}, it will mean that for such $\xi$  the perpetual functional exists.

\begin{lemma}\label{lemmm}
Let the jump kernel be symmetric, $a(x)=a(-x)$, has the second moment,
 $$\int_{\X}|x|^2a(x)\,{\mathrm{d}x}<\infty,$$ and its Fourier transform $\hat{a}$ is integrable. Then  $G_0(x)$ exists for any $x\in\X$, is integrable   and bounded.
 \end{lemma}
 \begin{proof} Consider the Fourier image of the jump kernel
 $\hat{a}(k)= \int_{\X} e^{-i(k,y)} a(y)\,{\mathrm{d}y}.$
Then
 $\hat{a}(0) =1,\;\; |\hat{a}(k)| < 1, k\neq 0, $ $\hat{a}(k)$ is real-valued and $\hat{a}(k) \to 0, k\to \infty.$
 Furthermore, the Fourier transform of any finite sum $\sum_{k=1}^n a_k(x)$, $n\in \N$ equals
 $$
 \hat{a}(k)\frac{1-\hat{a}^n(k)}{1-\hat{a}(k)}.
 $$ Keeping that in mind, and remembering that $\hat{a}^n$ is bounded, in order to establish that  the Fourier representation for $G_0$ has the form
 $$
G_0 (x)= \frac{1}{(2\pi)^d} \int_{\X} \frac{ \hat{a}(k) e^{i(k,x)}}{1-\hat{a}(k)}\,\mathrm{d}k,
$$
it is sufficient to prove that the function $\frac{ \hat{a}(k) }{1-\hat{a}(k)}$ is integrable. The existence of the 2nd moment of the jump kernel $a$ implies  the non-degeneracy of the covariance matrix, consequently elliptic in the sense that   at zero
 $|1-\hat{a}(k)|\ge C|k|^2$ for some $C>0$. Combined with the boundedness of  $\hat{a}$ it implies integrability of $\frac{ \hat{a}(k) }{1-\hat{a}(k)}$ in some ball around zero. Integrability outside this ball follows from integrability of $\hat{a}$.
\end{proof}

\subsection{Particular Models}

 There are several particular classes of jump kernels
for which Lemma \ref{lemmm} holds, see \cite{KS} for details.

\begin{example}[Gauss kernels]

Assume that the jump kernels has the following form:

\begin{equation*}
 a(x)= C_1\exp\left\{-\frac{ b |x|^2}{2}\right\}.
\end{equation*}
Then $\hat{a}(k)= C_1\exp\left\{-\frac{   |k|^2}{2b}\right\},$ therefore all conditions of Lemma  \ref{lemmm} hold.
 \end{example}

\begin{example}[Exponential tails]
Assume that
\begin{equation}
\label{exp}
a(x)\leq C\exp(-\delta|x|).
\end{equation}
It is proved in \cite{KS} that
for the kernel (\ref{exp}) and $d\geq 3$ it holds that
$$
G_0(x)\leq A\exp(-B|x|)
$$
with certain $A,B>0$.
 \end{example}

\subsection{Markov Processes}
\label{sec:Markov-processes}
Let  $X(t), t\ge 0$ be a Markov process in $\X$ starting from the point $x\in\X$.
A standard way to define a homogeneous  Markov process is to give the probability
$P_t(x,B)$ of the transition from the point $x\in \X$  to the set
$B\subset \X$ in   time $t>0$.  In some cases we have
$$
P_t(x,B)=\int_{B} p_t(x,y)\,{\mathrm{d}y},
$$
where $p_t(x,y)$ is the density of the transition probability. In any case, formally applying Fubini theorem, we obtain
\begin{equation}
\begin{gathered}\label{markov}
E\left[\int_0^\infty f(X(t))\,{\mathrm{d}t}\right]=\int_0^\infty E[f(X(t))]\,{\mathrm{d}t}=\int_0^\infty (T_tf)(x)\,{\mathrm{d}t}=\int_0^\infty \int_{\X}f(y)P_t(x,{\mathrm{d}y})\,{\mathrm{d}t}\\=
\int_{\X}f(y)\mathcal{G}(x, {\mathrm{d}y}),\end{gathered}\end{equation}
  where $\mathcal{G}(x,A)=\int_0^\infty P_t(x,A)\,{\mathrm{d}t}$ is the Green measure of the process $X$, see \cite{KS}. If the density of the transition probability exists, then we can consider the Green function $$g(x,y) =\int_0^\infty p_t(x,y)\,{\mathrm{d}t},$$ and in this case formally

\begin{equation} \label{markov1}
E\left[\int_0^\infty f(X(t))\,{\mathrm{d}t}\right]=\int_{\X}f(y)g(x,y)\,{\mathrm{d}y}.
\end{equation}

\begin{remark}
\begin{enumerate}
\item In Section \ref{sec:Examples} we have considered examples for which the Green function does exists, namely the Brownian motion, fractional Brownian motion and compound Poisson process. Therefore, the
right-hand side of \eqref{markov1} is well defined and the perpetual integral functional $\int_0^\infty f(X(t))\,{\mathrm{d}t}$ exits with probability $1$.
\item In general, we may consider Markov processes with no independent increments (as fBm) with uniformly elliptic generator. For this class of processes it was shown in \cite{Aron} and \cite{Grigor} that the density of transition probability admits two-sided bounds of Gaussian type, therefore for $f\in CL(\X)$, $\int_{\X}f(y)g(x,y)\,{\mathrm{d}y}$  is also well defined and the perpetual integral functional is finite a.s.
\end{enumerate}
\end{remark}

\subsection*{Acknowledgments}

We are thankful to Prof. Georgiy Shevchenko who proposed  how to prove non-constant property in Theorem \ref{theo1}. 
This work has been partially supported by Center for Research in Mathematics
and Applications (CIMA) related with the Statistics, Stochastic Processes
and Applications (SSPA) group, through the grant UIDB/MAT/04674/2020
of FCT-Funda{\c c\~a}o para a Ci{\^e}ncia e a Tecnologia, Portugal.

\end{document}